\documentclass[a4paper, fleqn, oneside]{amsart}
\pdfoutput=1
\usepackage{geometry}
\usepackage[utf8]{inputenc}
\usepackage[T1]{fontenc}
\usepackage[colorlinks]{hyperref}
\usepackage[capitalise]{cleveref}
\usepackage{microtype}
\usepackage{enumitem}
\usepackage{mathrsfs}
\hypersetup{
  linkcolor=[rgb]{0.3,0.3,0.6},
  citecolor=[rgb]{0.2, 0.6, 0.2},
  urlcolor=[rgb]{0.6, 0.2, 0.2}
}

\usepackage{mathtools, amsthm, amsfonts, amssymb}
\allowdisplaybreaks
\usepackage{braket}
\usepackage{commath}

\usepackage{accents}

\theoremstyle{plain}
\newtheorem{theorem}{Theorem}[section]
\newtheorem*{theorem*}{Theorem}
\newtheorem{proposition}[theorem]{Proposition}
\newtheorem{lemma}[theorem]{Lemma}

\newtheorem{corollary}[theorem]{Corollary}
\newtheorem{claim}[theorem]{Claim}

\theoremstyle{definition}

\newtheorem*{problem*}{Problem}
\newtheorem{problem}[theorem]{Problem}
\newtheorem{remark}[theorem]{Remark}
\newtheorem{example}[theorem]{Example}

\newcommand{\slfrac}[2]{\left.#1\middle/#2\right.}

\DeclarePairedDelimiter\floor{\lfloor}{\rfloor}

\DeclareMathOperator{\supp}{supp}

\DeclareMathOperator{\rank}{R}
\DeclareMathOperator{\subrank}{Q}

\DeclareMathAccent{\wtilde}{\mathord}{largesymbols}{"65}
\DeclareMathOperator{\asymprank}{\underaccent{\wtilde}{R}}
\DeclareMathOperator{\asympsubrank}{\underaccent{\wtilde}{Q}}

\newcommand{\FF}{\mathbb{F}}

\newcommand{\QQ}{\mathbb{Q}}
\newcommand{\NN}{\mathbb{N}}

\newcommand{\GHZ}{\mathrm{GHZ}}

\newcommand{\ZZ}{\mathbb{Z}}
\newcommand{\RR}{\mathbb{R}}

\DeclareMathOperator{\type}{type}

\DeclareMathOperator{\Span}{Span}

\newcommand{\defin}[1]{\emph{#1}}

\let\leqx\leqslant

\newcommand{\doasympleqx}{%
  \hbox{\ooalign{%
    \noalign{\kern.25ex}
    $\leqslant$\cr
    \noalign{\kern1.25ex}
    \smash{$\sim$}\cr
  }}%
}

\newcommand{\doasympdomleq}{%
  \hbox{\ooalign{%
    \noalign{\kern.25ex}
    $\preccurlyeq$\cr
    \noalign{\kern1.25ex}
    \smash{$\sim$}\cr
  }}%
}

\newcommand{\doasympasympdomleq}{%
  \hbox{\ooalign{%
    \noalign{\kern.25ex}
    $\preccurlyeq$\cr
    \noalign{\kern1.25ex}
    \smash{$\sim$}\cr
    \noalign{\kern0.5ex}
    \smash{$\sim$}\cr
  }}%
}

\DeclareMathOperator{\slicerank}{slicerank}

\let\leqx\leqslant

\newcommand{\doasympgeqx}{%
  \hbox{\ooalign{%
    \noalign{\kern.25ex}
    $\geqslant$\cr
    \noalign{\kern1.25ex}
    \smash{$\sim$}\cr
  }}%
}

\newcommand{\doasympasympleqx}{%
  \hbox{\ooalign{%
    \noalign{\kern.25ex}
    $\leqx$\cr
    \noalign{\kern1.25ex}
    \smash{$\sim$}\cr
    \noalign{\kern0.5ex}
    \smash{$\sim$}\cr
  }}%
}

\newcommand{\Vperpt}{\abs[1]{(V^\perp)_t}}

\newcommand{\fnc}{\frac{16c^2}{n^2} + \biggl(\frac{e\ln(2)c}{n}\biggr)^{\ln(2)c}}

\title[Asymptotic induced matching number of hypergraphs]{The asymptotic induced matching number\\ of hypergraphs: balanced binary strings}

\author{Srinivasan Arunachalam}
\address{Center for Theoretical Physics, Massachusetts Institute of Technology,
77 Massachusetts Ave, 6-304, 
Cambridge, MA 02139, USA }
\email{arunacha@mit.edu}
\author{Péter Vrana }
\address{Department of Geometry, Budapest University of Technology and Economics, Egry József~u.~1., 1111 Budapest, Hungary\vspace{-0.5em}}
\address{MTA-BME Lend\"ulet Quantum Information Theory Research Group} 
\email{vranap@math.bme.hu}
\author{Jeroen Zuiddam}
\address{Institute for Advanced Study, 1 Einstein Drive, Princeton, NJ 08540, USA}
\email{jzuiddam@ias.edu}
\date{\today}

\begin{document}

\begin{abstract}
We compute the asymptotic induced matching number of the $k$-partite $k$-uniform hypergraphs whose edges are the $k$-bit strings of Hamming weight $k/2$, for any large enough even number~$k$.
Our lower bound relies on the higher-order extension of the well-known Coppersmith--Winograd method from algebraic complexity theory, which was proven by Christandl, Vrana and Zuiddam.
Our result is motivated by the study of the power of this method as well as of the power of the Strassen support functionals (which provide upper bounds on the asymptotic induced matching number), and the connections to questions in  tensor theory, quantum information theory and theoretical computer science.

Phrased in the language of tensors, as a direct consequence of our result, we determine the asymptotic subrank of any tensor with support given by the aforementioned hypergraphs. 
In the context of quantum information theory, our result amounts to an asymptotically optimal~$k$-party stochastic local operations and classical communication (slocc) protocol for the problem of distilling GHZ-type entanglement from a subfamily of Dicke-type entanglement. 

\end{abstract}

\maketitle

\section{Introduction}\label{intro}

\subsection{Problem}\label{intro:result}
We study in this paper an asymptotic parameter of $k$-partite $k$-uniform hypergraphs: the asymptotic induced matching number.
For $k \in \NN$, a \defin{$k$-partite $k$-uniform hypergraph}, or \defin{$k$-graph} for short, is a tuple of finite sets~$V_1, \ldots, V_k$ together with a subset $\Phi$ of their cartesian product:
\[
\Phi \subseteq V_1\times \cdots \times V_k.
\]
Whenever possible we will leave the vertex sets $V_i$ implicit and refer to the $k$-graph by its edge set~$\Phi$.
For any $k \in \NN$ we use the notation~$[k] \coloneqq \{1, 2, \ldots, k\}$.
Let~$\Phi$ be a $k$-graph.
We say a subset~$\Psi$ of  $\Phi$ is \defin{induced} if~$\Psi = \Phi \cap (\Psi_1 \times \cdots \times \Psi_k)$ where for each~$i \in [k]$ we define the marginal set~$\Psi_i \coloneqq \{a_i : a \in \Psi\}$. 
We call %
$\Psi$ 
a \defin{matching} if any two distinct elements~$a,b \in \Psi$ are distinct in all~$k$ coordinates, that is,~$\forall i \in [k] : a_i \neq b_i$. %
The \defin{subrank}\footnote{The term subrank originates from an analogous parameter in the theory of tensors, see \cref{tensors}.} or  \defin{induced matching number}~$\subrank(\Phi)$ %
is defined as the size of the largest subset $\Psi$ of $\Phi$ that is an induced matching, that is,
\[
\subrank(\Phi) \coloneqq \max \{ \abs[0]{\Psi} : \Psi \subseteq \Phi, \Psi = \Phi \cap (\Psi_1 \times \cdots \times \Psi_k), \forall a\neq b \in \Psi\,\, \forall i \in [k]\,\, a_i \neq b_i\}.
\]

For example, consider the 3-graph $\Phi = \{(1,1,1), (2,2,2), (3,3,3)\} \subseteq [3]\times[3]\times[3]$. Here $\Phi$ is itself an induced matching, and so $\subrank(\Phi) = 3$. Next, let $\Phi = \{(1,1,1), (2,2,2), (3,3,3), (1,2,3)\}$. Now the subset $\{(1,1,1), (2,2,2)\} \subseteq \Phi$ is an induced matching and there is no larger induced matching in $\Phi$, and so $\subrank(\Phi) = 2$.

We define the \defin{Kronecker product} of two $k$-graphs $\Phi \subseteq V_1 \times \cdots \times V_k$ and $\Psi \subseteq W_1\times \cdots \times W_k$ as the~$k$-graph
\begin{align*}
\Phi \boxtimes \Psi &\coloneqq \bigl\{ \bigl((a_1, b_1), \ldots, (a_k, b_k)\bigr) : a \in \Phi, b \in \Psi \bigr\}\\ &\subseteq (V_1 \times W_1) \times \cdots \times (V_k \times W_k),
\end{align*}
and we naturally define the power $\Phi^{\boxtimes n} = \Phi \boxtimes \cdots \boxtimes \Phi$.
The \defin{asymptotic subrank} or 
the \defin{asymptotic induced matching number} of the $k$-graph $\Phi$ is defined as 
\[
\asympsubrank(\Phi) \coloneqq \lim_{n \to \infty} \subrank(\Phi^{\boxtimes n})^{1/n}.
\]
This limit exists and equals the supremum $\sup_{n\in \NN} \subrank(\Phi^{\boxtimes n})^{1/n}$ by Fekete's lemma (see, e.g.,~\cite[No.~98]{MR1492447}). %

We study the following basic question:

\begin{problem}\label{mainproblem}
Given $\Phi$ what is the value of $\asympsubrank(\Phi)$?
\end{problem}

A priori, for $\Phi \subseteq V_1 \times \cdots \times V_k$ we have the upper bound $\subrank(\Phi) \leq \min_i \abs[0]{V_i}$ and therefore holds that $\asympsubrank(\Phi) \leq \min_i \abs[0]{V_i}$, since $\abs[0]{V_i^{\times n}} = \abs[0]{V_i}^n$.

\cref{mainproblem} has been studied for several families of $k$-graphs, in several different contexts: the cap set problem \cite{MR3583358,tao,kleinberg2016growth, norin2016distribution, pebody2016proof}, approaches to fast matrix multiplication~\cite{strassen1991degeneration,MR3631613,blasiak2017groups,sawin2017bounds}, arithmetic removal lemmas \cite{lovasz2018lower,fox2018polynomial}, property testing~\cite{fu_et_al:LIPIcs:2014:4730,Haviv2017}, quantum information theory~\cite{vrana2015asymptotic, Vrana2017}, and the general study of asymptotic properties of tensors~\cite{sawin,christandl2016asymptotic,christandl2017universal}.
We finally mention the related result of Ruzsa and Szemerédi which says that the largest subset~$E \subseteq \binom{n}{2}$ such that $(E \times E \times E) \cap \{ (\{a,b\}, \{b, c\}, \{c, a\}) : a,b,c \in [n] \}$ is a matching, has size $n^{2 - o(1)} \leq \abs[0]{E} \leq o(n^2)$ when $n$ goes to infinity \cite{ruzsa1978triple}, see also \cite[Equation 2]{alon2006extremal}.

\subsection{Result} 
We solve \cref{mainproblem} for a family of $k$-graphs that are structured but nontrivial.
For $k \geq n$ let~$\lambda = (\lambda_1, \ldots, \lambda_n) \vdash k$ be an integer partition of $k$ with $n$ nonzero parts, that is,~$\lambda_1 \geq \lambda_2 \geq \cdots \geq \lambda_n > 0$ and $\sum_{i=1}^n \lambda_i = k$.
We define the $k$-graph
\[
\Phi_\lambda \coloneqq \{ s \in [n]^k : \type(s) = \lambda \}
\]
where the expression $\type(s) = \lambda$ means that $s$ is a permutation of the $k$-tuple
\[
(\underbrace{1, \ldots, 1}_{\lambda_1}, \underbrace{2, \ldots, 2}_{\lambda_2}, \ldots, \underbrace{n, \ldots, n}_{\lambda_n}).
\]
For example, the partition $\lambda = (1,1) \vdash 2$ corresponds to the 2-graph
\[
\Phi_{(1,1)} = \{(2,1),(1,2)\} \subseteq [2] \times [2]
\]
and the partition $\lambda = (2,2)\vdash 4$ corresponds to the 4-graph
\begin{align*}
\Phi_{(2,2)} &= \{(2,2,1,1),(2,1,2,1),(2,1,1,2),(1,2,2,1),(1,2,1,2),(1,1,2,2)\} \subseteq [2]^{\times 4}.
\end{align*}

It was shown in \cite{christandl2016asymptotic} %
that $\asympsubrank(\Phi_{(k-1, 1)}) %
= 2^{H((1-1/k, 1/k))}$ for every $k \in \NN_{\geq 3}$ where $H$ is the Shannon entropy in base 2. %
As a natural continuation of that work we study~$\asympsubrank(\Phi_{(k/2,k/2)})$ for even $k \in \NN$.
Since $\Phi_{(k/2,k/2)} \subseteq [2]^{\times k}$ we have $\asympsubrank(\Phi_{(k/2,k/2)}) \leq 2$. Clearly, the 2-graph $\Phi_{(1,1)}$ is itself a matching, and so $\asympsubrank(\Phi_{(1,1)}) = 2$. It was shown in~\cite{christandl2016asymptotic} that also~$\asympsubrank(\Phi_{(2,2)}) = 2$.
Our new result is the following extension:

\begin{theorem}\label{main}
Let $k \in \NN_{\geq2}$ be even and large enough. Then
$\asympsubrank(\Phi_{(k/2,k/2)}) = 2$.
\end{theorem}
In other words, we prove that for every large enough even $k \in \NN_{\geq2}$ there is an induced matching~$\Psi\subseteq \Phi_{(k/2,k/2)}^{\boxtimes n}$ %
of size $\abs[0]{\Psi} =2^{n - o(n)}$ when $n$ goes to infinity. 

Moreover, we numerically verified that~$\asympsubrank(\Phi_{(k/2,k/2)}) = 2$ also holds for all even $k \leq 2000$.
We conjecture that $\asympsubrank(\Phi_{(k/2,k/2)}) = 2$ for all even $k$.
More generally, we conjecture (cf.~\cite{vrana2015asymptotic} and \cite[Question~1.3.3]{christandl2016asymptotic}) that $\log_2 \asympsubrank(\Phi_\lambda)$ equals the Shannon entropy of the probability distribution obtained by normalising the partition~$\lambda$. We will discuss further motivation and background in \cref{motiv}.

\subsection{Methods}

We prove \cref{main} by applying the higher-order Coppersmith--Wi\-no\-grad~(CW) method from \cite{christandl2016asymptotic} to the $k$-graph $\Phi_{(k/2,k/2)}$. This method is an extension of the work of  %
Coppersmith and Winograd \cite{coppersmith1987matrix} and Strassen~\cite{strassen1991degeneration} from the case $k = 3$ to the case $k\geq 4$. 
It provides a construction of large induced matchings in $k$-graphs via the probabilistic method, and we prove \cref{main} by analysing the size of these induced matchings.

\begin{theorem}[Higher-order CW method \cite{christandl2016asymptotic}]\label{cw}
Let $\Phi \subseteq V_1\times \cdots \times V_k$ be a nonempty $k$-graph for which there exist injective maps $\alpha_i : V_i \to \ZZ$ such that for all $a \in \Phi$ the equality
\[
\alpha_1(a_1) + \cdots + \alpha_k(a_k) = 0
\]
holds.
For any $R \subseteq \Phi \times \Phi$ let $r(R)$ be the rank over $\QQ$ of the $\abs[0]{R}\times k$ matrix with rows %
\[
\{ \alpha(x) - \alpha(y) : (x,y) \in R\},
\]
where $\alpha(x) \coloneqq (\alpha_1(x_1), \ldots, \alpha_k(x_k))\in \ZZ^k$.
Then
\begin{equation}\label{lb}
\log_2 \asympsubrank(\Phi) \geq \max_{P\in \mathscr{P}} \Bigl( H(P) - (k-2) \max_{R\in \mathscr{R}} \frac{\max_{Q\in \mathscr{Q}_{R, (P_1, \ldots, P_k)}} H(Q) - H(P)}{r(R)}  \Bigr)
\end{equation}
where the parameters $P$, $R$ and $Q$ are taken over the following domains:
\begin{itemize}
\item $\mathscr{P}$ is the set of probability distributions on $\Phi$ %
\item $\mathscr{R}$ is the set of subsets of $\Phi \times \Phi$ that are not a subset of $\{(x,x) : x \in \Phi\}$ and moreover satisfy $\exists i \in [k]\, \forall (x,y) \in R\colon x_i = y_i$
\item $\mathscr{Q}_{R, (P_1, \ldots, P_k)}$ is the set of probability distributions on $R \subseteq \Phi \times \Phi$ %
with marginal distributions equal to $P_1, \ldots, P_k, P_1, \ldots, P_k$ respectively. %
\end{itemize}
Here for $P \in \mathscr{P}$ we denote by $P_1, \ldots, P_k$ the marginal probability distributions of $P$ on the components~$V_1, \ldots, V_k$ respectively, and $H$ denotes Shannon entropy.
\end{theorem}

Let $\lambda\vdash k$ be any integer partition of $k$ with $n$ nonzero parts. We can apply 
\cref{cw} to the $k$-graph $\Phi = \Phi_\lambda$ as follows.
For every $a \in \Phi_\lambda$ the equality
\begin{equation}\label{eqsum}
\sum_{\smash{i=1}}^k a_i = \sum_{\smash{j=1}}^n j \lambda_j
\end{equation}
holds, since the element $j$ occurs $\lambda_j$ times in $a$.
Let $\alpha_1, \ldots, \alpha_{k-1}$ be identity maps $\ZZ\to \ZZ$ and let $\alpha_k : \ZZ\to\ZZ : x \mapsto x - \sum_{j=1}^{\smash{n}} j \lambda_j$.
Then, because of \eqref{eqsum},
$\forall a \in \Phi_\lambda \colon \alpha_1(a_1) + \cdots + \alpha_k(a_k) = 0$.
(Note that with this choice of maps $\alpha_1, \ldots, \alpha_k$ we have that $\alpha(x) - \alpha(y)$ equals $x - y$ for every $(x,y) \in R$.)
Therefore \cref{cw} can be applied to obtain a lower bound on $\asympsubrank(\Phi_\lambda)$ for any partition $\lambda$. The difficulty now lies in evaluating the right-hand side of \eqref{lb}.

Let us return to the case $\lambda = (k/2,k/2)$.
To prove \cref{main} via \cref{cw} we will show for every large enough even $k \in \NN$ and $\Phi = \Phi_{(k/2,k/2)}$ that the right-hand side of \eqref{lb} is at least~2, using the aforementioned choice of injective maps $\alpha_1, \ldots, \alpha_k$.
In \cref{counting} we prove that this follows from the following statement, which may be of interest on its own.
\begin{theorem}\label{cl1}
For any large enough even $k \in \NN_{\geq4}$ and subspace $V \subseteq \{x \in \FF_2^k : x_k = 0\} \subseteq \FF_2^k$ the inequality
\begin{equation}\label{lhsth14}
\abs[1]{\bigl\{ (x,y) \in \FF_2^k \times \FF_2^k : \abs[0]{x} = \abs[0]{y} = \tfrac{k}{2},\, x-y \in V \bigr\}} \leq \smash{\binom{k-1}{k/2}^{\!\frac{\dim_{\FF_2}\!(V)}{k-2} + 1}}
\end{equation}
holds. Here $\abs[0]{x}$ denotes the Hamming weight of $x \in \FF_2^k$.
\end{theorem}

In~\cref{largensmalldim} we prove \cref{cl1} for low-dimensional $V$ by carefully splitting the left-hand side of \eqref{lhsth14} into two parts and upper bounding these parts. In \cref{largenlargedim} we prove \cref{cl1} for high-dimensional $V$ using Fourier analysis, Krawchouk polynomials and the Kahn--Kalai--Linial (KKL) inequality \cite{kahn1988influence}.
We thus prove \cref{cl1} and hence \cref{main}. While in our current proof the tools for the low- and high-dimensional cases are used complementarily, it may be possible that the full \cref{main} can be proven by cleverly using only the low-dimensional tools or only the high-dimensional tools.

\subsection{Motivation and background}\label{motiv} %

Our original motivation to study the asymptotic induced matching number of $k$-graphs comes from a connection to the study of asymptotic properties of tensors. In fact, the interplay in this connection goes both directions.
The purpose of this section is to discuss the asymptotic study of tensors and the connection with the asymptotic induced matching number.
Reading this section is not required to understand the rest of the paper. %

\subsubsection{Asymptotic rank and asymptotic subrank of tensors}\label{tensors}

The asymptotic study of tensors is a field of its own that started with the work of Strassen \cite{strassen1987relative, strassen1988asymptotic, strassen1991degeneration} in the context of fast matrix multiplication. We begin by introducing two fundamental asymptotic tensor parameters: asymptotic rank and asymptotic subrank.

Let $\FF$ be a field.
Let $a \in \FF^{n_1} \otimes \cdots \otimes \FF^{n_k}$ and $b \in \FF^{m_1} \otimes \cdots \otimes \FF^{m_k}$ be $k$-tensors. We write $a \leq b$ if there are linear maps $A_i : \FF^{m_i} \to \FF^{n_i}$ for $i \in [k]$ such that $a = (A_1 \otimes \cdots \otimes A_k)(b)$. 
For~$n \in \NN$ let $\{e_n : j \in [n]\}$ be the standard basis of~$\FF^n$.
For~$n \in \NN$ define the $k$-tensor
\[
\langle n \rangle \coloneqq \sum_{i=1}^n e_i \otimes \cdots \otimes e_i \in (\FF^n)^{\otimes k}.
\]
The rank of the $k$-tensor $a$ is defined as $\rank(a) \coloneqq \min\{ n \in \NN : a \leq \langle n \rangle\}$.
The subrank of the $k$-tensor $a$ is defined as 
\begin{equation}\label{Qdef}
\subrank(a) \coloneqq \max \{n \in \NN : \langle n \rangle \leq a\}.
\end{equation}
One can think of tensor rank as a measure of the complexity of a tensor, namely the ``cost'' of the tensor in terms of the diagonal tensors $\langle n\rangle$. It has been studied in several contexts, see, e.g.,~\cite{burgisser1997algebraic,landsberg2012tensors}. In this language, the subrank is the ``value'' of the tensor in terms of $\langle n\rangle$ and as such is the natural companion to tensor rank. It has its own applications, which we will elaborate on after having discussed the asymptotic viewpoint.

Writing~$a$ and $b$ in the standard basis as $a = \sum_i a_i\, e_{i_1} \otimes \cdots \otimes e_{i_k}$, $b = \sum_j b_j\, e_{j_1} \otimes \cdots \otimes e_{j_k}$,
the tensor Kronecker product $a\boxtimes b$ is the $k$-tensor defined by 
\[
a \boxtimes b \coloneqq \sum_{i,j} a_i b_j\,\, (e_{i_1} \otimes e_{j_1}) \otimes \cdots \otimes (e_{i_k}\otimes e_{j_k}) \in (\FF^{n_1} \otimes \FF^{m_1}) \otimes \cdots \otimes (\FF^{n_k}\otimes \FF^{m_k}).
\]
In other words, the $k$-tensor $a \boxtimes b$ is the image of the $2k$-tensor $a \otimes b$ under the natural regrouping map $\FF^{n_1} \otimes \cdots \otimes \FF^{n_k} \otimes \FF^{m_1} \otimes \cdots \otimes \FF^{m_k} \to (\FF^{n_1} \otimes \FF^{m_1}) \otimes \cdots \otimes (\FF^{n_k} \otimes \FF^{m_k})$.
The asymptotic rank of $a$ is defined as $\asymprank(a) \coloneqq \lim_{n\to\infty} \rank(a^{\boxtimes n})^{1/n}$ and
the asymptotic subrank of $a$ is defined as $\asympsubrank(a) \coloneqq \lim_{n\to\infty} \subrank(a^{\boxtimes n})^{1/n}$. These limits exist and equal the infimum $\inf_n \rank(a^{\boxtimes n})^{1/n}$ and the supremum $\sup_n \subrank(a^{\boxtimes n})^{1/n}$, respectively. This follows from Fekete's lemma and the fact that~$\rank(a\boxtimes b) \leq \rank(a)\rank(b)$ and $\subrank(a\boxtimes b) \geq \subrank(a) \subrank(b)$. 

Tensor rank is known to be hard to compute~\cite{HASTAD1990644} (the natural tensor rank decision problem is NP-hard).
Not much is known about the complexity of computing subrank, asymptotic subrank and asymptotic rank. 
It is a long-standing open problem in algebraic complexity theory to compute the asymptotic rank of the matrix multiplication tensor. The asymptotic rank of the matrix multiplication tensor corresponds directly to the asymptotic algebraic complexity of matrix multiplication. The asymptotic subrank of 3-tensors also plays a central role in the context of matrix multiplication, for example in recent work on barriers for upper bound methods on the asymptotic rank of the matrix multiplication tensor \cite{christandl2018barriers, alman2018limits}. As another example, in combinatorics, the resolution of the cap set problem \cite{MR3583358,tao} can be phrased in terms of the asymptotic subrank of a well-chosen 3-tensor, cf.~\cite{christandl2016asymptotic}, via the general connection to the asymptotic induced matching number that we will review now.

The subrank of $k$-tensors as defined in \eqref{Qdef} and the subrank of $k$-graphs as defined in \cref{intro:result} are related as follows.
For any $k$-tensor $a = \sum_i a_i\, e_{i_1} \otimes \cdots \otimes e_{i_k} \in \FF^{n_1} \otimes \cdots \otimes \FF^{n_k}$ we define the $k$-graph $\supp(a)$ as the support of $a$ in the standard basis:
\[
\supp(a) \coloneqq \{i \in [n_1]\times \cdots \times [n_k] : a_i \neq 0\}.
\]
It is readily ve\-ri\-fied that the subrank of the $k$-graph $\supp(a)$ is at most the subrank of the $k$-tensor~$a$, that is,~$\subrank(\supp(a)) \leq \subrank(a)$.
The reader may also verify directly that $\supp(a \boxtimes b) = \supp(a) \boxtimes \supp(b)$. Therefore, the asymptotic subrank of the support of $a$ is at most the asymptotic subrank of the $k$-tensor~$a$, that is,
\begin{equation}\label{eqa}
\asympsubrank(\supp(a)) \leq \asympsubrank(a).
\end{equation}
We can read~\eqref{eqa} in two ways. On the one hand, given any $k$-tensor $a$ we may find lower bounds on~$\asympsubrank(a)$ by finding lower bounds on $\asympsubrank(\supp(a))$. On the other hand, given any $k$-graph $\Phi \subseteq [n_1] \times \cdots \times [n_k]$ the asymptotic subrank $\asympsubrank(\Phi)$ is upper bounded by~$\asympsubrank(a)$ for any tensor $a \in \FF^{n_1} \otimes \cdots \otimes \FF^{n_k}$ (over any field $\FF$) with support equal to~$\Phi$, that is,%
\begin{equation}\label{Phiub}
\asympsubrank(\Phi) \leq \min_{\textnormal{field $\FF$}}\,\, \min_{\substack{a \in \FF^{n_1} \otimes\cdots \otimes \FF^{n_k} :\\ \supp(a) = \Phi}}\, \asympsubrank(a).
\end{equation}
We do not know whether the inequality in \eqref{Phiub} can be strict. We will discuss these two directions in the following two sections.

\subsubsection{Upper bounds on asymptotic subrank of $k$-tensors}\label{Qub}
Let us focus on the task of finding upper bounds on the asymptotic subrank of $k$-tensors. One natural strategy is to construct maps ${\phi : \{\textnormal{$k$-tensors over $\FF$}\} \to \RR_{\geq0}}$ that are sub-multiplicative under the tensor Kronecker product~$\boxtimes$, normalised on~$\langle n\rangle$ to $n$, and monotone under $\leq$, that is,~for any $k$-tensors~$a$ and~$b$ and for any~$n \in \NN$:
\begin{gather}
\phi(a\boxtimes b) \leq \phi(a) \phi(b)\label{p1}\\
\phi(\langle n\rangle) = n\label{p2}\\
a \leq b \Rightarrow \phi(a) \leq \phi(b).\label{p3}
\end{gather} 
The reader verifies directly that for any such map $\phi$ the inequality $\asympsubrank(a) \leq \phi(a)$ holds. 

Strassen in \cite{strassen1991degeneration}, motivated by the study of the algebraic complexity of matrix multiplication, introduced an infinite family of maps
\begin{equation*}\label{suppf}
\zeta^\theta : \{\textnormal{$k$-tensors over $\FF$}\} \to \RR_{\geq0}
\end{equation*}
parametrised by probability vectors $\theta\in\RR_{\geq0}^k$, $\sum_{i=1}^k\theta_i = 1$.
The maps $\zeta^\theta$ are called the upper support functionals. We will not define them here.
Strassen proved that each map $\zeta^\theta$ satisfies conditions~\eqref{p1},~\eqref{p2} and \eqref{p3}. Thus
\begin{equation}\label{zetaub}
\asympsubrank(a) \leq \min_\theta \zeta^\theta(a).
\end{equation}

Tao, motivated by the study of the cap set problem, proved in \cite{tao} that subrank is upper bounded by a parameter called slice rank, that is,~$\subrank(a) \leq \slicerank(a)$. We do not define slice rank here. While slice rank is easily seen to be normalised on~$\langle n\rangle$ and monotone under~$\leq$,  slice rank is not sub-multiplicative (see, e.g.,~\cite{christandl2017universal}). However, it still holds that 
\[
\asympsubrank(a) \leq \liminf_{n\to\infty} \slicerank(a^{\boxtimes n})^{1/n}.
\]
It turns out  \cite{sawin, christandl2017universal} that 
\[
\limsup_{n\to\infty} \slicerank(a^{\boxtimes n})^{1/n} \leq \min_\theta \zeta^\theta(a).
\]
No examples are known for which this inequality is strict.
It is known that for so-called oblique tensors holds $\limsup_{n\to\infty} \slicerank(a^{\boxtimes n})^{1/n} = \min_\theta \zeta^\theta(a)$ \cite{christandl2017universal}.

\subsubsection{Lower bounds on asymptotic subrank of $k$-graphs}
We now consider the task of finding lower bounds on the asymptotic subrank of $k$-graphs. 
For $k=3$ the CW method introduced by Coppersmith and Winograd \cite{coppersmith1987matrix} and extended by Strassen \cite{strassen1991degeneration} gives the following. 
Let $\Phi \subseteq V_1 \times V_2 \times V_3$ be a 3-graph for which there exist injective maps $\alpha_i : V_i \to \ZZ$ such that~$\forall a\in \Phi\colon \alpha_1(a_1) + \alpha_2(a_2) + \alpha_3(a_3) = 0$. Then
\begin{equation}\label{cw3}
\log_2 \asympsubrank(\Phi) \geq \max_{P \in \mathscr{P}} \min_{i\in[3]} H(P_i)
\end{equation}
where $\mathscr{P}$ is the set of probability distributions on $\Phi$. 
The inequality
\[
\log_2 \asympsubrank(\Phi) \leq \max_{P \in \mathscr{P}} \min_i H(P_i),
\]
follows from using \eqref{eqa} and using the support functionals as upper bound on the asymptotic subrank of tensors. Thus, the CW method is optimal whenever it can be applied.

\cref{cw} extends the CW method from $k = 3$ to higher-order tensors, that is,~$k\geq 4$. Contrary to the situation for $k=3$, the lower bound produced by \cref{cw} is not known to be tight.

\subsubsection{Type tensors}\label{sec:type}
As an investigation of the power of the higher-order CW method (\cref{cw}) and of the power of the support functionals (\cref{Qub}) we study the asymptotic subrank of the following family of tensors and their support.
While we do not have any immediate ``application'' for these tensors, we feel that they provide enough structure to make progress while still showing interesting behaviour.

Let $\lambda \vdash k$ be an integer partition of $k$ with~$n$ nonzero parts. 
Recall the definition of the $k$-graph~$\Phi_\lambda$ from \cref{intro:result}.
We define the tensor $T_\lambda$ as the $k$-tensor with support $\Phi_\lambda$ and all nonzero coefficients equal to 1, that is,
\[
T_\lambda \coloneqq \sum_{\mathclap{s \in \Phi_\lambda}} e_{s_1} \otimes \cdots \otimes e_{s_k} \in (\FF^{n})^{\otimes k}.
\]
In general, it follows from \eqref{eqa} and evaluating the right-hand side of \eqref{zetaub} for $a=T_\lambda$ and the uniform~$\theta = (1/k, \ldots, 1/k)$ that
\[
\asympsubrank(\Phi_\lambda) \leq \asympsubrank(T_\lambda) \leq 2^{H(\lambda/k)}.
\]
It was shown in \cite{christandl2016asymptotic} %
that 
\[
\asympsubrank(\Phi_{(k-1, 1)}) = \asympsubrank(T_{(k-1, 1)}) = 2^{H((1-1/k, 1/k))}
\]
for every $k \in \NN_{\geq 3}$ using \cref{cw}. (The same result was essentially obtained in~\cite{Haviv2017}.) In~\cite{christandl2016asymptotic} it was moreover shown that 
\[
\asympsubrank(\Phi_{(2, 2)}) = \asympsubrank(T_{(2, 2)}) = 2
\]
using \cref{cw}. %
As mentioned before, our main result (\cref{main}) is that for any large enough even~$k \in \NN_{\geq2}$ holds
\begin{equation}\label{eq:mainrepeat}
\asympsubrank(\Phi_{(k/2, k/2)}) = \asympsubrank(T_{(k/2, k/2)}) = 2.
\end{equation}
We conjecture that \eqref{eq:mainrepeat} holds for all even $k \in \NN$. We numerically verified this up to $k \leq 2000$. More generally we conjecture that $\asympsubrank(\Phi_\lambda) = \asympsubrank(T_\lambda) = 2^{H(\lambda/k)}$ holds for all partitions $\lambda\vdash k$, where~$H$ denotes the Shannon entropy and $\lambda/k$ denotes the probability vector $(\lambda_1/k, \ldots, \lambda_n/k)$.

In quantum information theory, the tensors $T_{(m, n)}$, when normalized, correspond to so-called Dicke states (see \cite{dicke1954coherence, stockton2003characterizing, vrana2015asymptotic}, and, e.g., \cite{bartschi2019deterministic}). Namely, in quantum information language, Dicke states are $(m+n)$-partite pure quantum states given by
\[
D_{(m,n)} \coloneqq \frac{1}{\sqrt{\binom{m+n}{m}}}T_{(m, n)} = \frac{1}{\sqrt{(m+n)!}} \sum_{\pi \in S_{m+n}} \pi \bigl(\ket{0}^{\otimes m} \otimes  \ket{1}^{\otimes n}\bigr)
\]
where the sum is over all permutations $\pi$ of the $k = m+n$ parties.
Roughly speaking, our result, \cref{main}, amounts to an asymptotically optimal~$k$-party stochastic local operations and classical communication (slocc) protocol for the problem of distilling GHZ-type entanglement from a subfamily of the Dicke states. More precisely, letting $\GHZ = \tfrac{1}{\sqrt{2}}(\ket{0}^{\otimes k} + \ket{1}^{\otimes k})$ be the $k$-party GHZ state, \cref{main} says that for $k$ large enough the maximal rate $\beta$ such that $n$ copies of $D_{(k/2,k/2)}$ can be transformed via slocc to $\beta n - o(n)$ copies of $\GHZ$ equals 1 when $n$ goes to infinity, that is,
\[
(D_{(k/2,k/2)})^{\otimes n} \xrightarrow{\textnormal{slocc}} \GHZ^{\otimes n - o(n)}
\]
and this rate is optimal.

\section{Reduction to counting}\label{counting}

We now begin working towards the proof of \cref{main}. The goal of this section is to reduce  \cref{main} to \cref{cl1} by applying \cref{cw}. %

\begin{lemma}\label{red}
\cref{cl1} implies \cref{main}.
\end{lemma}

\begin{proof}%
We will use the higher-order CW method \cref{cw} to show that \cref{cl1} implies \cref{main}.
Let $\Phi = \Phi_{(k/2, k/2)} = \{ x \in \{0,1\}^k : \abs[0]{x} = k/2\}$.
Let $\alpha_1, \ldots, \alpha_{k-1}$ be the identity map $\ZZ\to\ZZ$ and let $\alpha_k: \ZZ\to\ZZ : x \mapsto x - k/2$. 
With this definition of $\alpha$ we have for all $a \in \Phi$ satisfied the condition $\sum_i \alpha_i(a_i) = 0$ from \cref{cw}.
As in the statement of \cref{cw}, for~$R \in \mathscr{R}$ let $r(R)$ be the dimension of the $\QQ$-vector space 
\[
\Span_\QQ \{\alpha(x) - \alpha(y) : (x,y) \in R\} = \Span_\QQ \{ x - y : (x,y) \in R\}.
\]
Let $P$ be the uniform distribution on $\Phi$.
Then \cref{cw} gives
\begin{align*}
\log_2 \asympsubrank(\Phi) &\geq H(P) - (k-2) \max_{R\in \mathscr{R}}\frac{\max_{Q\in\mathscr{Q}_{R, (P_1, \ldots, P_k)}} H(Q) - H(P)}{r(R)}\\
&= \log_2 \binom{k}{k/2} - (k-2) \max_{R\in\mathscr{R}} \frac{\max_{Q\in\mathscr{Q}_{R, (P_1, \ldots, P_k)}} H(Q) - \log_2 \binom{k}{k/2}}{r(R)},
\end{align*}
For any $Q\in \mathscr{Q}_{R, (P_1, \ldots, P_k)}$ we have that $H(Q)$ is at most the Shannon entropy of the uniform distribution on $R$. %
We thus obtain
\begin{equation}\label{qeq}
\log_2 \asympsubrank(\Phi) \geq \log_2 \binom{k}{k/2} - (k-2) \max_{R\in\mathscr{R}} \frac{\log_2 \abs[0]{R} - \log_2 \binom{k}{k/2}}{r(R)}.
\end{equation}

It remains to upper bound the maximisation over $R\in \mathscr{R}$ in \eqref{qeq}.
We define the set 
\[
\Phi' = \{ x \in \{0,1\}^{k-1} : \abs[0]{x} = k/2-1\}.
\]
For $R \in \mathscr{R}$ let $r_2(R)$ be the dimension of the $\FF_2$-vector space
\[
\Span_{\FF_2} \{\alpha(x) - \alpha(y) : (x,y) \in R\} = \Span_{\FF_2} \{ x - y : (x,y) \in R\}.
\]%
By assumption \cref{cl1} is true. %
This means
\begin{align*}
\forall R' \subseteq \Phi'^{\times 2} \quad \log_2 \abs[0]{R'} &\leq \Bigl( \frac{r_2(R')}{k-2} + 1 \Bigr) \log_2 \binom{k-1}{k/2-1}\\
&= \frac{r_2(R')}{k-2} \log_2 \binom{k-1}{k/2-1} + \log_2 \binom{k-1}{k/2-1}\\
&= \frac{r_2(R')}{k-2} \log_2 \binom{k-1}{k/2-1} + \log_2 \tfrac12 \binom{k}{k/2}
\end{align*}
that is
\begin{equation}\label{eq1}
\forall R' \subseteq \Phi'^{\times 2} \quad \log_2 (2 \abs[0]{R'}) \leq \frac{r_2(R')}{k-2} \log_2 \binom{k-1}{k/2-1} + \log_2 \binom{k}{k/2}.
\end{equation}
For any $R \in \mathscr{R}$ there
is a subset $R' \subseteq \Phi'^{\times 2}$ with $\abs[0]{R} \leq 2\abs[0]{R'}$ and $r_2(R) = r_2(R')$.
Namely, one constructs $R'$ as follows. Without loss of generality $\forall (x,y) \in R\colon x_1 = y_1$. For every $(x,y) \in R$, if $x_1 = y_1 = 1$, then add $((x_2, \ldots, x_k),(y_2, \ldots, y_k))$ to $R'$, and if $x_1 = y_1 = 0$, then add the negated tuple $((1,\ldots, 1) - (x_2, \ldots, x_k), (1,\ldots, 1) - (y_2, \ldots, y_k))$ to $R'$.
Therefore, \eqref{eq1} implies
\begin{align*}
\forall R\in \mathscr{R}\quad \log_2\abs[0]{R} &\leq \frac{r_2(R)}{k-2} \log_2 \binom{k-1}{k/2-1} + \log_2 \binom{k}{k/2}\\
& = \frac{r_2(R)}{k-2} \Bigl( \log_2 2 \binom{k-1}{k/2-1}^2 - \log_2 \binom{k}{k/2} \Bigr) + \log_2 \binom{k}{k/2}
\end{align*}
that is
\[
\forall R\in \mathscr{R}\quad \log_2 \abs[0]{R} - \log_2 \binom{k}{k/2} \leq \frac{r_2(R)}{k-2} \Bigl( \log_2 2 \binom{k-1}{k/2-1}^2 - \log_2 \binom{k}{k/2}\Bigr)
\]
that is
\begin{equation}\label{eq2}
\forall R\in\mathscr{R}\quad \frac{\log_2 \abs[0]{R} - \log_2 \binom{k}{k/2}}{r_2(R)} \leq \frac{ \log_2 2 \binom{k-1}{k/2-1}^2 - \log_2 \binom{k}{k/2}}{k-2}.
\end{equation}

Combining \eqref{eq2} with \eqref{qeq} and using $r_2(R) \leq r(R)$ gives
\begin{align*}
\log_2 \asympsubrank(\phi) &\geq \log_2 \binom{k}{k/2} -  \biggl(\log_2 2\binom{k-1}{k/2-1}^{\!\!2} - \log_2 \binom{k}{k/2}\biggr)\\
&= \log_2 2\binom{k-1}{k/2-1} - \log_2 2\binom{k-1}{k/2-1}^{\!\!2} + \log_2 2\binom{k-1}{k/2-1}\\
&= \log_2 \binom{k-1}{k/2-1} - 2\log_2 \binom{k-1}{k/2-1} + \log_2 \binom{k-1}{k/2-1} + 1\\
&= 1.
\end{align*}
This proves the lemma.
\end{proof}

\section{Case: low dimension}\label{largensmalldim}

To prove \cref{main} it remains to prove \cref{cl1}. Our proof of \cref{cl1} is divided into two cases. In this section we prove the low-dimensional case.

\begin{theorem}\label{thm:mainthmsmalln}
For any even $k \in \NN_{\geq4}$ and subspace $V \subseteq \{x \in \FF_2^k : x_k = 0\} \subseteq \FF_2^k$ such that $\dim_{\FF_2}(V) \leq 11k/12$, the inequality
\[
\abs[1]{\bigl\{ (x,y) \in \FF_2^k \times \FF_2^k : \abs[0]{x} = \abs[0]{y} = \tfrac{k}{2},\, x-y \in V \bigr\}} \leq \smash{\binom{k-1}{k/2}^{\!\frac{\dim_{\FF_2}\!(V)}{k-2} + 1}}
\]
holds. %
\end{theorem}

We set up some notation.
Let $k\in 2\mathbb{N}$ and $\Phi=\{x\in\mathbb{F}_2^k\mid \abs[0]{x}=k/2\}$. %
We will think of $\mathbb{F}_2^{k-1}$ as the subspace where the last component is $0$.
We want to prove:		
		for any $V\le\mathbb{F}_2^{k-1}\le\mathbb{F}_2^k$ the inequality
		\begin{equation}\label{eq:desired}
		|R|\le\binom{k-1}{k/2}^{\frac{r}{k-2}+1}
		\end{equation}
		holds for all $r\leq \frac{11k}{12}$, where $R=\{(x,y)\in\Phi^2\mid x-y\in V,\, x_k=y_k=0\}$ and $r=\dim_{\mathbb{F}_2}V$.
The proof is divided into three claims. The first claim is trivial:

\begin{claim}
Inequality \eqref{eq:desired} holds when $r = 0$.
\end{claim}
\begin{proof}
One verifies directly that \eqref{eq:desired} becomes an equality when $r = 0$.
\end{proof}

	We prepare to deal with $r \geq 2$.
	Without loss of generality, we may assume that every vector in~$V$ has even weight.
To upper bound $\abs[0]{R}$ we introduce the function
	\begin{equation}
	f(k,m)=\begin{cases}
	\binom{m}{m/2}\binom{k-m-1}{(k-m)/2} & \text{if $m$ is even and $0\le m\le k-2$}  \\
	0 & \text{otherwise}
	\end{cases}
	\end{equation}
	which counts the number of pairs $(x,y)\in\Phi^2$ such that $x-y$ is an arbitrary but  fixed vector with  Hamming weight $m$. %
	This function has the following properties.
	
	\begin{proposition}\label{props}
	\leavevmode
		\begin{enumerate}
			\item For any even $0<m<k$ holds $f(k,m)=f(k,k-m)$.
			\item $f(k,m)$ strictly decreases in $m$ for even $0\le m\le k/2$.
			\item $f(k,0)=\binom{k-1}{k/2-1}=\binom{k-1}{k/2}$.
			\item $f(k,0)\ge f(k,k-2)=f(k,2)\ge f(k,k-4)=f(k,4)\ge\cdots.$

		\end{enumerate}
	\end{proposition}
	\begin{proof}
Claim (3) one verifies directly.
For (1) we verify that
\begin{align*}
f(k,k-m) &= \binom{k-m}{(k-m)/2}\binom{m-1}{m/2-1}\\
&= 2\binom{k-m-1}{(k-m)/2-1} \frac12\binom{m}{m/2}\\
&= f(k,m).%
\end{align*}
For (2) %
we verify that
\begin{align*}
\frac{f(k,m)}{f(k,m+2)} &= \frac{\binom{m}{m/2}\binom{k-m}{(k-m)/2}}{\binom{m+2}{(m+2)/2} \binom{k-m-2}{(k-m-2)/2}}\\
&=\slfrac{\displaystyle\frac{m!}{(\frac{m}{2}!)^2}\frac{(k-m)!}{(\frac{k-m}{2}!)^2}}{\displaystyle\frac{(m+2)!}{(\frac{m+2}{2}!)^2} \frac{(k-m-2)!}{(\frac{k-m-2}{2}!)^2 }}\\
&= \frac{(k-m)(k-m-1)}{(m+1)(m+2)}\frac{(\frac{m}{2}+1)^2}{(\frac{k-m}{2})^2}\\
&= \frac{m+2}{m+1}\,\frac{k-m-1}{k-m},
\end{align*}
which is $> 1$ when $(m+2)(k-m-1) > (m+1)(k-m)$, that is, when $k/2-2 \geq m$. %
Claim (4) follows from (1) and (2).
\end{proof}

	Using the definition of $f(k,m)$, we can write $|R|$ in \eqref{eq:desired} as follows: suppose $V$ has $a_m$ vectors of weight $m$, then
	\begin{equation}
	|R|=\sum_{m=0}^{k-1}a_m f(k,m).
	\end{equation}
	To get an upper bound on $|R|$, we fix some even $s\in \{2,\ldots,k/2\}$ and in the terms with $f(k,m)>f(k,s)$ we replace $a_m$ by $\binom{k-1}{m}$, while in the remaining terms we replace $f(k,m)$ by~$f(k,s)$. This gives, using \cref{props} (4),
	\begin{equation}
	\begin{split}
	|R|
	& \le f(k,0)+\sum_{\substack{m=2  \\ \text{$m$ even}}}^{s-2}\left[\binom{k-1}{k-m}+\binom{k-1}{m}\right]f(k,m)+f(k,s)\sum_{m=s}^{k-s}a_m  \\
	& \le f(k,0)+\sum_{\substack{m=2  \\ \text{$m$ even}}}^{s-2}\left[\binom{k-1}{m-1}+\binom{k-1}{m}\right]f(k,m)+2^rf(k,s)  \\
	& = \sum_{\substack{m=0  \\ \text{$m$ even}}}^{s-2}\binom{k}{m}f(k,m)+2^rf(k,s). %
	\end{split}
	\end{equation}
	Now our goal is to understand for which values of $k,r,s$ the inequality 
	\begin{align}
	\label{eq:upperboundukrs}
\sum_{\substack{m=0  \\ \text{$m$ even}}}^{s-2}\binom{k}{m}f(k,m)+2^rf(k,s) \le\binom{k-1}{k/2}^{\frac{r}{k-2}+1}
\end{align} holds. In particular, if for every $k$ and $r \leq 11k/12$, there exists such an $s$, then~\eqref{eq:desired} and hence~\cref{thm:mainthmsmalln} holds.
	
	First we replace \eqref{eq:upperboundukrs} by a stronger but simpler inequality. Divide both sides of~\eqref{eq:upperboundukrs} by $\binom{k-1}{k/2-1}$ and bound the right-hand side %
	from below as follows
	\begin{equation}\label{eq:binomialpower}
	2^r\left(\frac{\pi(k+1)}{2}\right)^{-\frac{r}{2(k-2)}} \leq \left(\frac{2^{k-1}}{\sqrt{\pi(k+1)/2}}\right)^{\frac{r}{k-2}}  \leq \binom{k-1}{k/2-1}^{\frac{r}{k-2}}.
	\end{equation}
Thus \eqref{eq:upperboundukrs} is implied by
	\begin{align}
	\label{eq:firstsimplification}
	\sum_{\substack{m=0  \\ \text{$m$ even}}}^{s-2}\frac{\binom{k}{m}f(k,m)}{\binom{k-1}{k/2-1}}+	\frac{2^r f(k,s)}{\binom{k-1}{k/2-1}} \leq  2^r\left(\frac{\pi(k+1)}{2}\right)^{-\frac{r}{2(k-2)}}
	\end{align}

	\begin{claim}
		Inequality~\eqref{eq:desired} holds for every $k\geq 27$, and 
		$
		r\in \{2,\ldots,\frac{k}{2\log k} \}.
		$
	\end{claim}

	\begin{proof}
		Let $s = 2$.
		The left-hand side of \eqref{eq:firstsimplification} equals
		\begin{equation}
		1+2^r\cdot 2\frac{\binom{k-3}{(k-2)/2}}{\binom{k-1}{k/2}}=1+2^r\frac{1}{2}\frac{k}{k-1}.
		\end{equation}
		Since $2^{-r}\le\frac{1}{4}$, we see that \eqref{eq:firstsimplification} is implied by %
		\begin{equation}
		\frac{1}{4}+\frac{1}{2}\frac{k}{k-1}\le\left(\frac{\pi(k+1)}{2}\right)^{-\frac{r}{2(k-2)}}.
		\end{equation}
		This is equivalent to 
\begin{equation}\label{eq111}
		r\leq 2(k-2)\frac{\log(\frac{1}{1/4+k/(2(k-1))})}{\log(\pi/2 \cdot (k+1))}.
\end{equation}
		We use that for $k$ large enough holds $\frac{1}{1/4+k/(2(k-1))}\geq 13/10$, %
		$2(k-2)\geq \frac{5}{3}k$, %
		and
\[
		\log(\pi/2 \cdot (k+1))\cdot \frac{3}{5} \cdot \frac{1}{\log(13/10)}\leq 2\log(k)
\]
to see that the right-hand side of \eqref{eq111} is at least $k/(2\log k)$.
	\end{proof}

	We now further simplify the left-hand side of~\eqref{eq:firstsimplification} via
	\begin{equation}
		\label{eq:upperboundterm1}
		\begin{split}
\frac{\binom{k}{m}f(k,m)}{\binom{k-1}{k/2-1}} &= \frac{\binom{k}{m}\binom{m}{m/2}\binom{k-m-1}{(k-m)/2}}{\binom{k-1}{k/2-1}}\\ & = \frac{k!m!(k-m-1)!(k/2-1)!(k/2)!}{m!(k-m)!\left(\frac{m}{2}!\right)^2\frac{k-m}{2}!\left(\frac{k-m}{2}-1\right)!(k-1)!}  \\
	& = \frac{k(k/2-1)!}{(k-m)\frac{m}{2}!\left(\frac{k-m}{2}-1\right)!}\binom{k/2}{m/2}  \\
	& = \frac{\frac{k}{2}(k/2-1)!}{\frac{k-m}{2}\frac{m}{2}!\left(\frac{k-m}{2}-1\right)!}\binom{k/2}{m/2}=\binom{k/2}{m/2}^2
	\end{split}
	\end{equation}
	and
	\begin{equation}
	\label{eq:upperboundterm2}
	\begin{split}
\frac{f(k,s)}{\binom{k-1}{k/2-1}} &= \binom{s}{s/2}\frac{\binom{k-s-1}{(k-s)/2}}{\binom{k-1}{k/2-1}}\\
  & = \binom{s}{s/2}\frac{(k-s-1)!\left(\frac{k}{2}-1\right)!\frac{k}{2}!}{\frac{k-s}{2}!\left(\frac{k-s}{2}-1\right)!(k-1)!}  \\
	& = 2^{-s}\binom{s}{s/2}\frac{\displaystyle 2^{s/2}\frac{\frac{k}{2}!}{\frac{k-s}{2}!}}{\displaystyle 2^{-s/2}\frac{(k-1)!}{(k-s-1)!}\frac{\left(\frac{k-s}{2}-1\right)!}{\left(\frac{k}{2}-1\right)!}}  \\
	& = 2^{-s}\binom{s}{s/2}\prod_{i=0}^{s/2-1}\frac{k-2i}{k-2i-1}  = 2^{-s}\binom{s}{s/2}\prod_{i=0}^{s/2-1}\left(1+\frac{1}{k-2i-1}\right).
	\end{split}
	\end{equation}
	We have the upper bound $\binom{s}{s/2}\le2^s\sqrt{\frac{2}{\pi s}}$. In the product of $s/2$ terms, each term is at least~$1$ and the largest term is the last one. Since $s\le k/2$, we can use $k-s-1\ge k/2-1$ to get 
	\begin{equation}
	1\le\prod_{i=0}^{s/2-1}\left(1+\frac{1}{k-2i-1}\right)\le\left(1+\frac{1}{k-s-1}\right)^{s/2}\le\left(1+\frac{1}{k/2-1}\right)^{k/4}\le 2
	\end{equation}
	for all $k\ge 4$. 
	Plugging in~\eqref{eq:upperboundterm1},\eqref{eq:upperboundterm2} into~\eqref{eq:firstsimplification}, we see that \eqref{eq:upperboundukrs} is implied by
	\begin{equation}
	\sum_{\substack{m=0  \\ \text{$m$ even}}}^{s-2}\binom{k/2}{m/2}^2+ 2^r\sqrt{\frac{8}{\pi s}}\le 2^r\left(\frac{\pi(k+1)}{2}\right)^{-\frac{r}{2(k-2)}},
	\end{equation}
	that is, \eqref{eq:upperboundukrs} is implied by 
	\begin{equation}\label{eq:secondsimplification}
	2^{-r}\sum_{\substack{m=0  \\ \text{$m$ even}}}^{s-2}\binom{k/2}{m/2}^2+\sqrt{\frac{8}{\pi s}}\le \left(\frac{\pi(k+1)}{2}\right)^{-\frac{r}{2(k-2)}}.
	\end{equation}
To further upper bound the left-hand side of~\eqref{eq:secondsimplification} we use the following lemma, which we will prove later.
	\begin{lemma}\label{lem:sumratio}
		For any even $k$ and $2\le s\le k/2$ the following inequality holds:
		\begin{equation}
		\frac{\displaystyle\sum_{\substack{m=0  \\ \text{$m$ even}}}^{s}\binom{k/2}{m/2}^2}{\displaystyle\sum_{\substack{m=0  \\  \text{$m$ even}}}^s\binom{k}{m}}\le \frac{4}{\sqrt{\pi}}\cdot\sqrt{\frac{k}{s(k-s)}}
		\end{equation}
	\end{lemma}
\begin{remark}
Numerics suggest that the optimal constant in the above inequality is $\sqrt{2/\pi}$ instead of $4/\sqrt{\pi}$.
\end{remark}

Assuming that $r$ satisfies
\begin{equation}
\label{eq:assumptionbetweenrands}
\sum_{\substack{m=0  \\  \text{$m$ even}}}^{s-2}\binom{k}{m}\le 2^r
\end{equation}
we have
\begin{align}
\label{eq:upperboundonlhsofsecondassumption}
\begin{aligned}
2^{-r}\sum_{\substack{m=0  \\ \text{$m$ even}}}^{s-2}\binom{k/2}{m/2}^2+\sqrt{\frac{8}{\pi s}} 
&\leq 2^{-r}\frac{4}{\sqrt{\pi}}\cdot\sqrt{\frac{k}{s(k-s)}}\displaystyle\sum_{\substack{m=0  \\  \text{$m$ even}}}^s\binom{k}{m}+\sqrt{\frac{8}{\pi s}}\\
&\leq \frac{4}{\sqrt{\pi}}\cdot\sqrt{\frac{k}{s(k-s)}}+\sqrt{\frac{8}{\pi s}} \\
&\leq \frac{4}{\sqrt{\pi}}\cdot\sqrt{\frac{2}{s}}+\sqrt{\frac{8}{\pi s}} = 3\sqrt{\frac{8}{\pi s}},
\end{aligned}
\end{align}
where the first inequality used Lemma~\ref{lem:sumratio}, the second inequality used~\eqref{eq:assumptionbetweenrands}, and the third inequality used $\frac{k}{k-s}\le 2$ (which holds, since $s\leq k/2$). 
Thus, assuming \eqref{eq:assumptionbetweenrands}, we have that \eqref{eq:secondsimplification} is implied by	
	\begin{equation}
	3\sqrt{\frac{8}{\pi s}}\le\left(\frac{\pi(k+1)}{2}\right)^{-\frac{r}{2(k-2)}}.
	\end{equation}
	In other words, if there is an $s\ge 24\ge\frac{72}{\pi}=22.9183...$ such that
	\begin{equation}
	\label{eq:thirdsimplification}
	\log\sum_{\substack{m=0  \\  \text{$m$ even}}}^{s-2}\binom{k}{m}\le r\le(k-2)\frac{\log s-\log\frac{72}{\pi}}{\log(k+1)+\log\frac{\pi}{2}},
	\end{equation}
	then~\eqref{eq:secondsimplification} holds. %
	We further upper bound the left-hand side of~\eqref{eq:thirdsimplification} by
	\begin{equation}
	\log\sum_{\substack{m=0  \\  \text{$m$ even}}}^{s-2}\binom{k}{m}\le\log\sum_{m=0}^{s-2}\binom{k}{m}\le k h\Bigl(\frac{s}{k}\Bigr).
	\end{equation}
	Hence \eqref{eq:thirdsimplification} is implied by
	\begin{align}
	\label{eq:fourthsimplification}
k h\Bigl(\frac{s}{k}\Bigr)\le r\le(k-2)\frac{\log s-\log\frac{72}{\pi}}{\log(k+1)+\log\frac{\pi}{2}}.
	\end{align}

\begin{claim}
Inequality~\eqref{eq:desired} holds for $k$ large enough and every $r\in \{\frac{k}{2\log k},\ldots,11k/12\}$.
\end{claim}

\begin{proof}
	Use the bound of~\eqref{eq:fourthsimplification} with $s=2\lfloor k^\beta/2\rfloor$ to get %
	that
	inequality~\eqref{eq:desired} holds for $\beta \in (0,1)$, $k\geq \max\{24^{1/\beta},2^{1/(1-\beta)}\}$, and 
\begin{equation}\label{lem:largen}
		kh(k^{\beta-1})\leq r\leq (k-2)\frac{\log(k^\beta-2)-\log\frac{72}{\pi}}{\log(k+1)+\log \frac{\pi}{2}}.
\end{equation}
Fix $\beta=1-\frac{2\log \log k}{\log k}$. For this choice of $\beta$, we have $k\geq 24^{1/\beta}$ for every $k\geq 3500$ and clearly~$k\geq 2^{1/(1-\beta)}$ for every $k\geq3$, thereby satisfying the requirements for \eqref{lem:largen}. Now observe that
		\begin{align}
		\label{eq:UBonkh()}
		kh(k^{\beta-1})=	kh\Big(\frac{1}{\log^2 k}\Big)\leq \frac{4k}{\log^2 k}\log \log k \leq \frac{k}{2\log k},
		\end{align} 
		where the first inequality uses the fact that for every $x\in (0,1/2]$ holds $h(x)\leq 2 x\log \frac{1}{x}$, and the second inequality holds for every~$k\geq 13\cdot 10^{12}$.
				 Next, for $k$ large enough
\begin{equation}
\begin{split}
		(k-2)\frac{\log(k^\beta-2)-\log\frac{72}{\pi}}{\log(k+1)+\log \frac{\pi}{2}} &\geq 	(k-2)\frac{\log(k^\beta-2)-\log\frac{72}{\pi}}{\log(2k)}\geq (k-2)\frac{\log(k^\beta-2)-{5}}{\log(2k)}\\ &\geq (k-2)\frac{\log k^\beta-{6}}{\log(2k)}.
\end{split}
\end{equation}
		For very large $k$, observe that
		\begin{align}
		\label{eq:LBonfrac}
		(k-2)\frac{\log k^\beta-{6}}{\log(2k)} %
		\geq \frac{11k}{12}.
		\end{align}
		Putting together \eqref{eq:LBonfrac} and \eqref{eq:UBonkh()} along with \eqref{lem:largen}, we prove the claim.
	\end{proof}

	\begin{proof}[\bfseries\upshape Proof of Lemma~\ref{lem:sumratio}]
	We will make use of the following variant of Stirling's formula (due to Robbins \cite{MR0069328}), valid for all positive integers $n$:
	\begin{equation}
	\sqrt{2\pi n}\frac{n^n}{e^n}e^{\frac{1}{12n+1}}<n!<\sqrt{2\pi n}\frac{n^n}{e^n}e^{\frac{1}{12n}}.
	\end{equation}
	First we bound the ratio of the individual terms (assuming $m\neq 0$) as
	\begin{equation}\label{eq:termratio}
	\begin{split}
	\frac{\displaystyle\binom{k/2}{m/2}^2}{\displaystyle\binom{k}{m}}
	& = \frac{\frac{k}{2}!^2m!(k-m)!}{\frac{m}{2}!^2\frac{k-m}{2}!^2k!}  \\
	& \le \frac{1}{\sqrt{2\pi}}\sqrt{\frac{\frac{k^2}{4}m(k-m)}{\frac{m^2}{4}\frac{(k-m)^2}{4}k}}\frac{\left(\frac{k}{2}\right)^km^m(k-m)^{k-m}}{\left(\frac{m}{2}\right)^m\left(\frac{k-m}{2}\right)^{k-m}k^k}  \\
	& \cdot\exp\left\{\frac{1}{3k}+\frac{1}{12m}+\frac{1}{12(k-m)}-\frac{2}{6m+1}-\frac{2}{6(k-m)+1}-\frac{1}{12k+1}\right\}  \\
	& \le \sqrt{\frac{2}{\pi}}\sqrt{\frac{k}{m(k-m)}},
	\end{split}
	\end{equation}
	since the third factor is $1$ and the argument of the exponential is negative if $2\le m\le\frac{k}{2}$.
	
	Now let us turn to the ratio of the sums. Let $0<c_1<2c_1<c_2<\frac{1}{2}$ be fixed constants. Assume first that $2\le s\le c_2k$. The denominator can be bounded from below by its last term, while the numerator can be bounded from above as
	\begin{align}
	\begin{aligned}
	\sum_{\substack{m=0  \\ \text{$m$ even}}}^{s}\binom{k/2}{m/2}^2 = \sum_{i=0}^{s/2}\binom{k/2}{i}^2   = \sum_{j=0}^{s/2}\binom{k/2}{s/2-j}^2  & \le \sum_{j=0}^{s/2}\binom{k/2}{s/2}^2\left(\frac{s}{k-s}\right)^{2j}  \\
	& \le \sum_{j=0}^{\infty}\binom{k/2}{s/2}^2\left(\frac{s}{k-s}\right)^{2j}  \\
	& = \binom{k/2}{s/2}^2\frac{(k-s)^2}{k(k-2s)}\\
	& \le \binom{k/2}{s/2}^2\frac{(1-c_2)^2}{1-2c_2},
	\end{aligned}
	\end{align}
	where in the first inequality we have used
	\begin{equation}
	\frac{\displaystyle\binom{k/2}{n}}{\displaystyle\binom{k/2}{n+1}}=\frac{n+1}{\frac{k}{2}-n}\le\frac{\frac{s}{2}-1+1}{\frac{k}{2}-\frac{s}{2}+1}\le\frac{s}{k-s}
	\end{equation}
	for $n+1\le s/2$. Combining with \eqref{eq:termratio} we arrive at the estimate
	\begin{equation}
	\frac{\displaystyle\sum_{\substack{m=0  \\ \text{$m$ even}}}^{s}\binom{k/2}{m/2}^2}{\displaystyle\sum_{\substack{m=0  \\  \text{$m$ even}}}^s\binom{k}{m}}\le \frac{1-c_2}{1-2c_2}\frac{\displaystyle\binom{k/2}{s/2}^2}{\displaystyle\binom{k}{s}}\le\frac{1-c_2}{1-2c_2}\sqrt{\frac{2}{\pi}}\sqrt{\frac{k}{s(k-s)}}.
	\end{equation}
	
	Now we turn to the case when $c_2k\le s\le k/2$. Split the sum in the numerator into two at~$m\approx c_1k$. For $m\le\lfloor c_1k\rfloor$ use the simple bound $\binom{k/2}{m/2}^2\le\binom{k}{m}$, while for $m\ge\lfloor c_1k\rfloor+1\ge c_1k$ use~\eqref{eq:termratio} to get
	\begin{equation}
	\frac{\displaystyle\binom{k/2}{m/2}^2}{\displaystyle\binom{k}{m}}\le\sqrt{\frac{2}{\pi}}\sqrt{\frac{k}{m(k-m)}}\le\sqrt{\frac{2}{\pi}}\frac{1}{\sqrt{k}}\frac{1}{\sqrt{c_1(1-c_1)}}.
	\end{equation}
	Introducing
	\begin{align}
	A  = \sum_{\substack{m=0  \\ \text{$m$ even}}}^{2\lfloor c_1k/2\rfloor}\binom{k}{m}, \qquad
	B  = \sum_{\substack{m=2\lfloor c_1k/2\rfloor+2  \\ \text{$m$ even}}}^{s}\binom{k}{m}.
	\end{align}
	The estimate
\begin{equation}
\begin{split}
	\frac{\displaystyle\sum_{\substack{m=0  \\ \text{$m$ even}}}^{s}\binom{k/2}{m/2}^2}{\displaystyle\sum_{\substack{m=0  \\  \text{$m$ even}}}^s\binom{k}{m}} &\le\frac{A+\sqrt{\frac{2}{\pi}}\frac{1}{\sqrt{k}}\frac{1}{\sqrt{c_1(1-c_1)}}B}{A+B}  =\frac{\sqrt{\frac{2}{\pi}}\frac{1}{\sqrt{k}}\frac{1}{\sqrt{c_1(1-c_1)}}+\frac{A}{B}}{1+\frac{A}{B}}  \\
	&\le \sqrt{\frac{2}{\pi}}\frac{1}{\sqrt{k}}\frac{1}{\sqrt{c_1(1-c_1)}}+\frac{A}{B}
\end{split}
\end{equation}
	follows.
	The ratio
	\begin{equation}
	\frac{\displaystyle\binom{k}{n}}{\displaystyle\binom{k}{n-1}}=\frac{k-n+1}{n}=\frac{k+1}{n}-1
	\end{equation}
	is monotonically decreasing in $n$, therefore, by induction
	\begin{equation}
	\frac{\displaystyle\binom{k}{b-t}}{\displaystyle\binom{k}{a-t}}\ge\frac{\displaystyle\binom{k}{b}}{\displaystyle\binom{k}{a}}
	\end{equation}
	whenever $a\le b$. Apply this with $a=2\lfloor c_1k/2\rfloor$, $b=s$ and $t=2\lfloor c_1k/2\rfloor-m$ to get
	\begin{equation}
	\begin{split}
	A = \sum_{\substack{m=0  \\ \text{$m$ even}}}^{2\lfloor c_1k/2\rfloor}\binom{k}{m}   
	= \sum_{\substack{m=0  \\ \text{$m$ even}}}^{2\lfloor c_1k/2\rfloor}\binom{k}{m+s-2\lfloor c_1k/2\rfloor}\frac{\binom{k}{m}}{\binom{k}{m+s-2\lfloor c_1k/2\rfloor}}\\
	 \le \sum_{\substack{m=0  \\ \text{$m$ even}}}^{2\lfloor c_1k/2\rfloor}\binom{k}{m+s-2\lfloor c_1k/2\rfloor}\frac{\binom{k}{2\lfloor c_1k/2\rfloor}}{\binom{k}{s}}  \\
	 = \frac{\binom{k}{2\lfloor c_1k/2\rfloor}}{\binom{k}{s}}\sum_{\substack{m=s-2\lfloor c_1k/2\rfloor  \\ \text{$m$ even}}}^{s}\binom{k}{m}\le\frac{\binom{k}{2\lfloor c_1k/2\rfloor}}{\binom{k}{s}}B,
	\end{split}
	\end{equation}
	that is,
	\begin{equation}
	\frac{A}{B}\le\frac{\binom{k}{2\lfloor c_1k/2\rfloor}}{\binom{k}{s}}\le 2^{k(h(c_1)-h(s/k))}\sqrt{8k\frac{s}{k}\left(1-\frac{s}{k}\right)}\le 2^{k(h(c_1)-h(c_2))}\sqrt{2k}.
	\end{equation}
	We now look for a constant $C$ that satisfies
	\begin{equation}
	\sqrt{\frac{2}{\pi}}\frac{1}{\sqrt{k}}\frac{1}{\sqrt{c_1(1-c_1)}}+2^{k(h(c_1)-h(c_2))}\sqrt{2k}\le C\cdot\sqrt{\frac{k}{s(k-s)}}
	\end{equation}
	when $c_2k\le s\le k/2$. Equivalently, we need
	\begin{equation}
	\sqrt{\frac{2}{\pi}}\sqrt{\frac{s}{k}\left(1-\frac{s}{k}\right)}\frac{1}{\sqrt{c_1(1-c_1)}}+\sqrt{2}\cdot 2^{k(h(c_1)-h(c_2))}k\sqrt{\frac{s}{k}\left(1-\frac{s}{k}\right)}\le C.
	\end{equation}
	Using $\sqrt{\frac{s}{k}\left(1-\frac{s}{k}\right)}\le\frac{1}{2}$ and that $2^{k(h(c_1)-h(c_2))}k$ has a global maximum at $k=\frac{1}{\ln 2}\frac{1}{h(c_2)-h(c_1)}$, an upper bound on the left-hand side is
	\begin{equation}
	\frac{1}{\sqrt{2\pi}}\frac{1}{\sqrt{c_1(1-c_1)}}+\frac{1}{\sqrt{2}e\ln 2}\frac{1}{h(c_2)-h(c_1)}.
	\end{equation}
	In particular, with $c_1=0.09711\ldots$ and $c_2=0.39252\ldots$ we get
	$C=2.25503\ldots<\frac{4}{\sqrt{\pi}}$.
\end{proof}

\section{Case:  high dimension}\label{largenlargedim}

Finally, in this section we consider the remaining high-dimensional case.
\begin{theorem}\label{largelargethm}\label{sj} %
For any large enough even $k \in \NN_{\geq4}$ and subspace $V \subseteq \{x \in \FF_2^k : x_k = 0\} \subseteq \FF_2^k$ such that $\dim_{\FF_2}(V) \geq 11(k-1)/12$, the inequality
\[
\abs[1]{\bigl\{ (x,y) \in \FF_2^k \times \FF_2^k : \abs[0]{x} = \abs[0]{y} = \tfrac{k}{2},\, x-y \in V \bigr\}} \leq \smash{\binom{k-1}{k/2}^{\!\frac{\dim_{\FF_2}\!(V)}{k-2} + 1}}
\]
holds. Here $\abs[0]{x}$ denotes the Hamming weight of $x \in \FF_2^k$.
\end{theorem}

\subsection{Preliminaries}
Our proof of \cref{sj} uses Fourier analysis on the Boolean cube $\FF_2^n = \{0,1\}^n$, the Krawchouk polynomials, a consequence of the KKL inequality and some elementary bounds for expressions involving binomial coefficients.

\subsubsection{Fourier transform}

For $z\in \{0,1\}^n$ define the function $\chi_z : \{0,1\}^n \to \RR$ by $\chi_z(x) = (-1)^{z\cdot x}$ with $z\cdot x = \sum_i z_i x_i$. These so-called \defin{characters} form an orthonormal basis for the space of functions $\{0,1\}^n \to \RR$ for the inner product $\langle f,g\rangle = \frac{1}{2^n} \sum_x f(x) g(x)$.
For a function $f : \{0,1\}^n \to \RR$ define $\widehat{f} : \{0,1\}^n \to \RR$ by $\widehat{f}(z) = \langle f,\chi_z\rangle = \frac{1}{2^n}\sum_{x} f(x) \chi_z(x)$. The function $\widehat{f}$ is the \defin{Fourier transform of $f$}. 
One verifies that for any functions $f,g : \{0,1\}^n \to \RR$ we have the identity
\begin{equation}\label{conv}
\sum_{x,y} f(x) f(y) g(x + y) = 2^{2n} \sum_z \widehat{f}(z)^2 \widehat{g}(z)
\end{equation}
with sums over $x,y \in \{0,1\}^n$ and $z\in \{0,1\}^n$.

\subsubsection{Krawchouk polynomials}
For $0 \leq k \leq n$ define the function 
\[
K_k^n : \{0,1\}^n \to \RR
\]
as the sum of the characters~$\chi_z$ with $z \in \{0,1\}^n$ and $|z| = k$, that is 
\[
K_k^n(x) = \sum_{|z|=k} \chi_z(x).
\]
The function $K_k^n(x)$ depends only on the Hamming weight $|x|$ and can thus be interpreted as a function on integers $0 \leq t \leq n$. This function may be written as $K^n_k(t) = \sum_{j=0}^k (-1)^j \binom{t}{j} \binom{n-t}{k-j}$ and this defines a real polynomial of degree~$k$, called the \defin{$k$th Krawchouk polynomial}. We will use the following expression for the ``middle'' Krawchouk polynomial for odd $n$.

\begin{lemma}[Proposition 4.4 in \cite{feinsilver2016sums}]\label{kraw} Let $n$ be odd and $t\in\{0, \ldots, n\}$. Then
\[
K^n_{\frac{n-1}{2}}(t) =  %
(-1)^{\floor{t/2}} \binom{n}{(n-1)/2} \frac{\displaystyle\binom{(n-1)/2}{\floor{t/2}}}{\displaystyle \binom{n}{t}}. %
\]
\end{lemma}
We will encounter the Krawchouk polynomials in the following way.
For any $0\leq k \leq n$ define the function $w^n_k : \{0,1\}^n \to \{0,1\}$ by $w^n_k(z) = [|z|=k]$. Then
\begin{equation}\label{what}
\widehat{w}^n_k(z) = \frac{1}{2^n}\sum_x w^n_k(x) (-1)^{x\cdot z} = \frac{1}{2^n} K_k^n(|z|).
\end{equation}

\subsubsection{KKL inequality}
Let $A\subseteq \{0,1\}^n$. The \defin{characteristic function} $f: \{0,1\}^n \to \{0,1\}$ of $A$ is defined by $f(x) = [x\in A]$. %
Now suppose $A$ is a linear subspace. Let $A^{\perp} \coloneqq \{y \in \{0,1\}^n : \textnormal{$y \cdot x = 0$ for all $x\in A$}\}$ be the orthogonal complement of $A$. The Fourier transform of $f$ is given by
\begin{equation}\label{ft}
\widehat{f}(z) = \frac{[z\in A^\perp]}{|A^\perp|}.
\end{equation}
Indeed, $\widehat{f}(z) = \tfrac{1}{2^n} \sum_{x\in A} (-1)^{x \cdot z}$ and, if $z \in A^\perp$, then this sum equals $\tfrac{1}{2^n} |A|$; if $z \not\in A^\perp$, say $x_0 \cdot z = 1$, then $\sum_{x\in A} (-1)^{x \cdot z} = \sum_{x \in A} (-1)^{(x + x_0)z} = (-1) \sum_{x\in A} (-1)^{x\cdot z}$ so the sum equals zero.

The following lemma is a consequence of the KKL inequality \cite{kahn1988influence} and can be found in~\cite{MR2847893}.

\begin{lemma}[KKL inequality]\label{kkl} Let $A \subseteq \{0,1\}^n$ be a non-empty subset. Let~$f$ be the characteristic function of $A$. Define $c= n - \log |A|$. For any integer $1\leq t \leq  \ln(2)c$ we have
\begin{align*}
\sum_{|z| = t} \widehat{f}(z)^2 &\leq \frac{1}{2^{2 c}} \biggl(\frac{2 e \ln (2) c}{t}\biggr)^t\\
\sum_{|z| = n-t} \widehat{f}(z)^2 &\leq \frac{1}{2^{2 c}} \biggl(\frac{2 e \ln (2) c}{t}\biggr)^t
\end{align*}
with sums over $z\in \{0,1\}^n$.
\end{lemma}

For any subset $A\subseteq \{0,1\}^n$ and integer $0\leq t \leq n$ we denote by $A_t$ the set of vectors in $A$ with Hamming weight $t$.

\begin{corollary}\label{kklcor}
Let $V \subseteq\{0,1\}^n$ be a subspace and define $c = n - \dim(V)$. For any integer $1\leq t\leq \ln(2)c$ we have the following upper bound on the number of vectors in $V^\perp$ with Hamming weight $t$ and $n-t$ respectively:
\begin{align*}
\Vperpt &\leq \biggl(\frac{2 e \ln (2) c}{t}\biggr)^t\\
\abs[1]{(V^\perp)_{n-t}} &\leq \biggl(\frac{2 e \ln (2) c}{t}\biggr)^t.
\end{align*}
\end{corollary}
\begin{proof} Let $f$ be the indicator function of $V$. Then, using \eqref{ft} and \cref{kkl} we get
\[
\Vperpt =  |V^\perp|^2\sum_{|z|=t}\biggl(\frac{[z\in V^\perp]}{|V^\perp|}\biggr)^2 =  2^{2c} \sum_{|z|=t} \widehat{f}(z)^2 \leq \biggl(\frac{2 e \ln (2) c}{t}\biggr)^t
\]
and the same for $\abs[1]{(V^\perp)_{n-t}}$.
\end{proof}
\begin{example} As mentioned in \cite{MR2847893} the following example shows that \cref{kklcor} is almost tight. Let $V\subseteq \{0,1\}^n$ be the $d$-dimensional subspace consisting of all bit strings that begin with $n-d$ zeros. Then $V^\perp$ is the space of bit strings that end with $d$ zeros. Let $c = n - \dim(V) = n-d$. Then we can directly compute the lower bound
\[
\biggl(\frac{c}{t}\biggr)^t \leq \binom{c}{t} = \binom{n-d}{t} = \abs[1]{(V^{\perp})_t}
\]
while \cref{kklcor} gives for $1\leq t \leq \ln(2)c$ that
\[
\abs[1]{(V^{\perp})_t} \leq \biggl(\frac{2 e \ln (2) c}{t}\biggr)^t.
\]

\end{example}

\subsubsection{Bounds involving binomial coefficients}

\begin{lemma}\label{bounds} %
Let $n$ be even. If $0\leq m\leq n/3$, then
\[
\frac{\binom{n/2}{m}}{\binom{n+1}{2m+1}} \leq 2\biggl(\frac{2m+1}{2(n-m+1)}\biggr)^{m+1}.
\]
If $1\leq m\leq (n+1)/3$, then
\[
\frac{\binom{n/2}{m}}{\binom{n+1}{2m}} \leq \biggl(\frac{m}{n-m+1}\biggr)^{m}.
\]
\end{lemma}
\begin{proof}
We expand the binomial coefficients as fractions of factorials:
\begin{align*}
\frac{\binom{n/2}{m}}{\binom{n+1}{2m+1}} &= \frac{(n/2)! (2m+1)! (n-2m)!}{m! (n/2-m)! (n+1)!} \\
&= \frac{(n/2)\cdots (n/2-m+1)}{(n+1)\cdots (n-m+2)} \cdot \frac{(2m+1)\cdots (m+1)}{(n-m+1)\cdots(n-2m+1)}\\
&\leq 2\biggl(\frac{2m+1}{2(n-m+1)}\biggr)^{m+1}
\end{align*}
where in the last inequality we upper bounded each of the first $m$ terms by~$1/2$ and each of the last $m+1$ terms by $(2m+1)/(n-m+1)$ using the assumption $m\leq n/3$. We do the same for the other inequality:
\begin{align*}
\frac{\binom{n/2}{m}}{\binom{n+1}{2m}} &= \frac{(n/2)! (2m)! (n+1-2m)!}{m! (n/2-m)! (n+1)!} \\
&= \frac{(n/2)\cdots (n/2-m+1)}{(n+1)\cdots (n-m+2)} \cdot \frac{(2m)\cdots (m+1)}{(n-m+1)\cdots(n-2m+2)}\\
&\leq \biggl(\frac{m}{n-m+1}\biggr)^{m}
\end{align*}
where in the last inequality we upper bounded each of the first $m$ terms by~$1/2$ and each of the last $m$ terms by $(2m)/(n-m+1)$ using the assumption $1\leq m\leq (n+1)/3$.
\end{proof}

\subsection{Proof of \cref{largelargethm}}

\begin{proof}[\upshape\bfseries Proof of \cref{largelargethm}] %
Let $n\geq 59$ be odd. Let $V\subseteq \{0,1\}^n$ be a subspace of dimension at least $11n/12$. We will prove that
\begin{equation}\label{eqtoprove}
\abs[1]{\bigl\{ (x,y) \in (\{0,1\}^n)^{\times 2} : |x| = |y| = \tfrac{n-1}{2}, x + y \in V\bigr\}} \leq \binom{n}{\frac{n-1}{2}}^{1 + \frac{\dim_{\FF_2}(V)}{n-1}}.
\end{equation}
This proves the theorem. To see this, in the theorem statement, set $k = n+1$, ignore the $(n+1)$th coordinate, and note that the size of $\bigl\{ (x,y) \in (\{0,1\}^n)^{\times 2} : |x| = |y| = \tfrac{n-1}{2}, x + y \in V\bigr\}$ equals the size of $\bigl\{ (x,y) \in (\{0,1\}^n)^{\times 2} : |x| = |y| = \tfrac{n+1}{2}, x + y \in V\bigr\}$ via the bijection that flips the bits of $x$ and $y$.

Let $f : \{0,1\}^n \to \{0,1\}$ be the characteristic function of $V$, that is, $f(x) = [x\in V]$. Recall that we defined the function $w_k^n : \{0,1\}^n \to \{0,1\}$ by $w^n_k(x) = [|x|=k]$. Using \eqref{conv} the left-hand side of \eqref{eqtoprove} can be rewritten as
\begin{align*}
&\abs[1]{\bigl\{ (x,y) \in (\{0,1\}^n)^{\times 2} : |x| = |y| = \tfrac{n-1}{2}, x + y \in V\bigr\}}\\
&= \sum_{x,y} w^n_{\frac{n-1}{2}}(x) w^n_{\frac{n-1}{2}}(y) f(x + y)\\
&= 2^{2n} \sum_{z} \widehat{w}^n_{\frac{n-1}{2}}(z)^2 \widehat{f}(z)
\end{align*}
with sums over $x,y\in \{0,1\}^n$ and $z \in \{0,1\}^n$. Since $\widehat{w}^n_k(z)=\frac{1}{2^n}K^n_k(|z|)$ (see~\eqref{what}) and $\widehat{f}(z) = \frac{1}{2^n} |V| \cdot [z \in V^\perp]$ (see \eqref{ft}) we have
\begin{align}\label{expr1}
 2^{2n} \sum_{z} \widehat{w}^n_{\frac{n-1}{2}}(z)^2 \widehat{f}(z) = \frac{|V|}{2^n} \sum_{z} K^n_{\frac{n-1}{2}}(|z|)^2 \, [z \in V^\perp].
\end{align}
Recall that $(V^\perp)_t$ denotes the subset of $V^\perp$ consisting of vectors with Hamming weight $t$.
We rewrite the right-hand side of \eqref{expr1} as a sum over the Hamming weight $t=|z| \in \{0, \ldots, n\}$. 
\begin{align}
&\frac{|V|}{2^n} \sum_{z} K^n_{\frac{n-1}{2}}(|z|)^2 \, [z \in V^\perp] = \frac{|V|}{2^n} \sum_{t} K^n_{\frac{n-1}{2}}(t)^2  \,\Vperpt. \label{eqc}%
\end{align}
By \cref{kraw} we have 
\[
K^n_{\frac{n-1}{2}}(t)^2 = 
\binom{n}{(n-1)/2}^2 \frac{\displaystyle\binom{(n-1)/2}{\floor{t/2}}^2}{\displaystyle \binom{n}{t}^2}
\]
which we use to rewrite \eqref{eqc} as
\begin{align}
&\frac{|V|}{2^n} \biggl( \sum_{\mathclap{\textnormal{$t$}}} K^n_{\frac{n-1}{2}}(t)^2 \,\Vperpt \biggr)\nonumber\\ %
&= \frac{|V|}{2^n} \binom{n}{\frac{n-1}{2}}^2\biggl( \sum_{\mathclap{\textnormal{$t$}}}\frac{\binom{(n-1)/2}{\floor{t/2}}^2}{ \binom{n}{t}^2} \,\Vperpt \biggr)\nonumber\\ %
&= \frac{|V|}{2^n} \binom{n}{\frac{n-1}{2}}^2\biggl( 1 + [1^n \in V^{\perp}] +  \sum_{\substack{1\leq t\leq n-1\\}} \frac{\binom{(n-1)/2}{\floor{t/2}}^2}{ \binom{n}{t}^2} \,\Vperpt \biggr). \label{eqtobound} %
\end{align}
We assumed that $\dim(V) \geq 11n/12$. Since the statement of the theorem is directly verified to be true when $\dim(V) = n-1$ %
we may in addition assume that $\dim(V) < n-1$. We define~$c = n - \dim(V)$. Then $2\leq c \leq n/12$.
Let
\[
f(n,c) \coloneqq \fnc%
\]
In \cref{lem1} and \cref{lem2} below we will prove the inequalities
\begin{gather}
\sum_{1\leq t\leq n-1} \frac{\binom{(n-1)/2}{\floor{t/2}}^2}{ \binom{n}{t}^2} \,\Vperpt  \leq f(n,c)\\
2+f(n,c) \leq 2^{c} \binom{n}{\frac{n-1}{2}}^{\frac{1-c}{n-1}}.
\end{gather}
These inequalities show that \eqref{eqtobound} is upper bounded as follows:
\begin{align*}
&\frac{|V|}{2^n} \binom{n}{\frac{n-1}{2}}^2\biggl( 1 + [1^n \in V^{\perp}] +  \sum_{\substack{1\leq t\leq n-1\\}} \frac{\binom{(n-1)/2}{\floor{t/2}}^2}{ \binom{n}{t}^2} \,\Vperpt \biggr)\\
&\leq\frac{|V|}{2^n} \binom{n}{\frac{n-1}{2}}^2 \bigl( 2 + f(n,c) \bigr)\\
&\leq\frac{|V|}{2^n} \binom{n}{\frac{n-1}{2}}^2 2^c \binom{n}{\frac{n-1}{2}}^{\frac{1-c}{n-1}}\\
&=\binom{n}{\frac{n-1}{2}}^{1 + \frac{\dim(V)}{n-1}}
\end{align*}
which proves the theorem.
\end{proof}

\begin{lemma}\label{lem1}
Let $n$ be odd.
For $2 \leq c \leq n/12$  such that $\dim(V) = n-c$ we~have
\[
\sum_{1\leq t\leq n-1} \frac{\binom{(n-1)/2}{\floor{t/2}}^2}{ \binom{n}{t}^2} \,\Vperpt \leq f(n,c).
\]
with
\[
f(n,c) \coloneqq \fnc. %
\]
\end{lemma}
\begin{proof}
We first upper bound the sum over $t \in [1, \floor{\ln(2)c}] \cup [n - \floor{\ln(2)c}, n-1]$ and afterwards the sum over the remaining $t$'s.
We use $\binom{(n-1)/2}{\floor{t/2}} = \binom{(n-1)/2}{\floor{(n-t)/2}}$ and then apply \cref{kklcor} to get
\begin{align}
&\sum_{t=1}^{\floor{\ln(2)c}} \frac{\binom{(n-1)/2}{\floor{t/2}}^2}{ \binom{n}{t}^2} \,\Vperpt  \,+ \sum_{t= n-\floor{\ln(2)c}}^{n-1} \frac{\binom{(n-1)/2}{\floor{t/2}}^2}{ \binom{n}{t}^2} \,\Vperpt\nonumber \\
&= \sum_{t=1}^{\floor{\ln(2)c}} \frac{\binom{(n-1)/2}{\floor{t/2}}^2}{ \binom{n}{t}^2} \Bigl(\,\Vperpt + \abs[1]{(V^\perp)_{n-t}}\Bigr)\nonumber\\
&\leq 2 \sum_{t=1}^{\floor{\ln(2)c}} \frac{\binom{(n-1)/2}{\floor{t/2}}^2}{ \binom{n}{t}^2} \biggl( \frac{2 e \ln(2) c}{t} \biggr)^t\label{eqeq}.
\end{align}
We upper bound the sum over even $t$ and the sum over odd $t$ separately.
For the even part we use $\floor{\ln(2)c}\leq c$, then use \cref{bounds} %
 and replace $t$ by $2t$ to get
\begin{align}
\sum_{\substack{t=1\\[0.5ex]\textnormal{even}}}^{\floor{\ln(2)c}} \frac{\binom{(n-1)/2}{t/2}^2}{ \binom{n}{t}^2}  \biggl(\frac{2 e \ln (2) c}{t}\biggr)^t 
&\leq \sum_{\substack{t=1\\[0.5ex]\textnormal{even}}}^{c} \frac{\binom{(n-1)/2}{t/2}^2}{ \binom{n}{t}^2}  \biggl(\frac{2 e \ln (2) c}{t}\biggr)^t \label{eqstar2}\\
&\leq \sum_{\substack{t=1\\[0.5ex]\textnormal{even}}}^{c} \biggl( \frac{t}{2n-t} \biggr)^t  \biggl(\frac{2 e \ln (2) c}{t}\biggr)^t\nonumber\\
&= \sum_{\substack{t=1\\[0.5ex]\textnormal{even}}}^{c} \biggl(\frac{2 e \ln (2) c}{2n-t}\biggr)^t\nonumber\\
&= \sum_{t=1}^{c/2} \biggl(\frac{e \ln (2) c}{n-t}\biggr)^{2t}\nonumber.
\end{align}
We upper bound the sum %
as follows, using $t\leq c/2$ and $c\leq n/12$:
\begin{align}
&\sum_{t=1}^{c/2} \biggl(\frac{(e \ln 2) c}{n-t}\biggr)^{2t} \leq \sum_{t=1}^{c/2} \biggl(\frac{(e \ln 2) c}{n-c/2}\biggr)^{2t} = \frac{4c^2(e \ln 2)^2 \bigl(1 - \bigl(\frac{ce \ln 4}{2n -c} \bigr)^c \bigr)}{4n^2-4cn-c^2(4(e \ln2)^2-1)}\nonumber\\
&\leq \frac{4c^2(e \ln 2)^2}{4n^2-4cn-4c^2(e \ln2)^2} \leq \frac{4c^2}{n^2} \biggl( \frac{(e \ln 2)^2}{4 - 1/3-(e \ln 2)^2/36} \biggr) \leq \frac{4c^2}{n^2}  \label{eqstar1}.
\end{align}
For the odd part we shift $t$ by 1 and use $\floor{\ln(2)c}\leq c-1$, then use \cref{bounds} %
to get
\begin{align}
\sum_{\substack{t=1\\[0.5ex]\textnormal{odd}}}^{\floor{\ln(2)c}} \frac{\binom{(n-1)/2}{(t-1)/2}^2}{ \binom{n}{t}^2}  \biggl(\frac{2 e \ln (2) c}{t}\biggr)^t
&= \sum_{\substack{t=1\\[0.5ex]\textnormal{even}}}^{c} \frac{\binom{(n-1)/2}{t/2-1}^2}{ \binom{n}{t-1}^2}  \biggl(\frac{2 e \ln (2) c}{t-1}\biggr)^{t-1} \nonumber\\
&\leq \sum_{\substack{t=1\\[0.5ex]\textnormal{even}}}^{c} 4\biggl(\frac{t-1}{2n-t} \biggr)^t   \biggl(\frac{2 e \ln (2) c}{t-1}\biggr)^{t} \biggl(\frac{t-1}{2 e \ln (2) c}\biggr) \nonumber
\end{align}
Next we use $t \leq \ln(2)c +1$ and $4\leq 2e$ and we replace~$t$ by $2t$ to get
\begin{align*}
\sum_{\substack{t=1\\[0.5ex]\textnormal{even}}}^{c} 4\biggl(\frac{t-1}{2n-t} \biggr)^t   \biggl(\frac{2 e \ln (2) c}{t-1}\biggr)^{t} \biggl(\frac{t-1}{2 e \ln (2) c}\biggr) &\leq
\sum_{\substack{t=1\\[0.5ex]\textnormal{even}}}^{c} \biggl(\frac{t-1}{2n-t} \biggr)^t   \biggl(\frac{2 e \ln (2) c}{t-1}\biggr)^{t}\\
&= \sum_{t=1}^{c/2} \biggl(\frac{e \ln (2) c}{n-t}\biggr)^{2t}
\end{align*}
which again we upper bound with \eqref{eqstar1}.  We conclude that \eqref{eqeq} is upper bounded by $16c^2/n^2$.

To upper bound the sum over the remaining $t$'s we use the inequalities
\[
\binom{\frac{n-1}{2}}{\floor{\frac{t}{2}}}^2 \leq \binom{n}{t},\qquad \biggl(\frac{n}{k}\biggr)^{k}\leq \binom{n}{k}.
\]
to get
\begin{align*}
&\sum_{t=\floor{\ln(2) c}+1}^{n-\floor{\ln(2)c}-1}  \frac{\binom{(n-1)/2}{\floor{t/2}}^2}{ \binom{n}{t}^2}  \,\Vperpt
\leq \sum_{t=\floor{\ln(2) c}+1}^{n-\floor{\ln(2)c}-1}  \frac{\abs[1]{(V^\perp)_t}}{ \binom{n}{t}} \\
&\leq \frac{1}{ \binom{n}{\floor{\ln(2)c}}} \sum_{t=\floor{\ln(2) c}}^{n-\floor{\ln(2)c}-1}   \,\Vperpt \leq \frac{\abs[1]{V^\perp}}{ \binom{n}{\floor{\ln(2)c}}}\\
& \leq 2^c \biggl(\frac{\ln(2) c}{n}\biggr)^{\ln(2)c} = \biggl(\frac{e \ln(2) c}{n}\biggr)^{\ln(2)c}.
\end{align*}
This finishes the proof.
\end{proof}

\begin{lemma}\label{lem2} For $n\geq 59$ odd and $2\leq c \leq n/12$ we have
\[
2+f(n,c) \leq 2^{c} \binom{n}{\frac{n-1}{2}}^{\frac{1-c}{n-1}}.
\]
with
\[
f(n,c) \coloneqq \fnc. %
\]
\end{lemma}
\begin{proof}
For odd $n$ we have $2^{n}/\sqrt{n} \geq \binom{n}{(n-1)/2}$ and thus
\[
2^{1+\frac{1-c}{n-1}}\sqrt{n}^{\frac{c-1}{n-1}} = 2^{c} \biggl( \frac{2^{n}}{\sqrt{n}} \biggr)^{\frac{1-c}{n-1}} \leq 2^{c} \binom{n}{\frac{n-1}{2}}^{\frac{1-c}{n-1}}.
\]
It is thus sufficient to show that for $n\geq 59$ and $2 \leq c \leq n/12$ we have $2 + f(n,c) \leq 2(\sqrt{n}/2)^{\frac{c-1}{n-1}}$.
One verifies that $2+f(n,2) \leq 2(\sqrt{n}/2)^{\frac{2-1}{n-1}}$ holds for every $n\geq 53$. We will show that for every $n\geq 59$ the function $f_n(c) = 2(\sqrt{n}/2)^{\frac{c-1}{n-1}} - (2 + f(n,c))$ is increasing in $c$ for $2\leq c \leq n/12$. %
We see that the derivative $\frac{\mathrm{d}}{\mathrm{d}c} f_n(c)$ equals
\[
\frac{\mathrm{d}}{\mathrm{d}c}f_n(c) = 2(\sqrt{n}/2)^{\frac{c-1}{n-1}} \frac{\ln(\sqrt{n}/2)}{n-1} - \frac{32c}{n^2} - g_n(c)
\]
with
\[
g_n(c) = %
\ln(2) \biggl(\frac{e \ln(2) c}{n}\biggr)^{\ln (2)c} \ln \Bigl(\frac{e^2\ln (2)c}{n}\Bigr).
\]
Using $c\leq n/12$ one can verify that $\ln (e^2\ln (2)c / n)\leq 0$ so that $g_n(c) \leq 0$. Moreover, using $c\leq n/12$, $n\geq 59$ and $(\sqrt{n}/2)^{\frac{c-1}{n-1}}\geq1$ one can verify that
\[
\frac{32c}{n^2} \leq \frac{32}{12n} = \frac{8/3}{n} \leq \frac{8/3}{n-1} \leq \frac{2\ln(\sqrt{n}/2)}{n-1}\leq \frac{2\ln(\sqrt{n}/2)}{n-1}(\sqrt{n}/2)^{\frac{c-1}{n-1}}.
\]
We conclude that $\frac{\mathrm{d}}{\mathrm{d}c} f_n(c) \geq 0$ which proves the lemma.
\end{proof}

\subsubsection*{\upshape\bfseries Acknowledgements}
SA is funded by the MIT--IBMWatson AI Lab under the project \emph{Machine Learning in Hilbert space}. This work was initiated when SA was a part of QuSoft, CWI and was supported by ERC Consolidator Grant QPROGRESS.
JZ thanks Florian Speelman, Pjotr Buys and Avi Wigderson for helpful discussions. This work was initiated when JZ was a part of QuSoft, CWI. This material is based upon work supported by the National Science Foundation under Grant No.~DMS-1638352 (JZ). %
This research was supported by the National Research, Development and Innovation Fund of Hungary within the Quantum Technology National Excellence Program (Project Nr.~2017-1.2.1-NKP-2017-00001) and via the research grants K124152, KH~129601 (PV). 

\newpage
\raggedright
\bibliographystyle{alphaurlpp}
\bibliography{all}
\end{document}